\documentclass{amsart}
\usepackage[latin1]{inputenc}

\usepackage{amssymb,amsmath}
\usepackage{amsthm}
\usepackage{enumerate}
\usepackage{url}

\theoremstyle{plain}
\newtheorem{example}{Example}

\newtheorem{definition}{Definition}

\newtheorem{problem}{Problem}
\newtheorem{theorem}{Theorem}

\newtheorem{lemma}[theorem]{Lemma}

\def\endproof{\hfill{\vrule height 4pt width 4pt depth 0pt}\par\smallbreak}

\def\d{{\rm dom}\hphantom{.}}

\def\2{{\bf 2}}
\def\P2{{\rm Par}(\2)}
\def\PmA{{\rm Par}^{(m)}(A)}
\def\PnA{{\rm Par}^{(n)}(A)}
\def\Pn2{{\rm Par}^{(n)}(2)}
\def\PA{{\rm Par}(A)}
\def\OA{{\rm O}(A)}
\def\O2{{\rm O}(\2)}
\def\OpnA{{\rm O}^{(n)}_{A}}

\def\pp{{\rm pPol}\,}
\def\PT{{\rm Pol}\,}

\def\st{~|~}

\def \imp{\Longrightarrow}

\def\va{{\vec a}}
\def\vb{{\vec b}}
\def\vv{{\vec v}}

\def\ve{{\vec e}}
\def\r{\rho}

\def\0{\vec{0}}
\def\1{\vec{1}}
\def\k{{\bf k}}
\def\JA{{\rm J}(A)}
\def\J2{{\rm J}(\2)}

\def\pne{{\pp(\ne)}}
\def\ple{{\pp(\le)}}
\def\maj{{\rm maj}}

\def\cF{{\rm F}}

\begin{document}
\title{Finitely generated maximal partial clones and their intersections}
\author{Miguel Couceiro}
\address[M. Couceiro]{Mathematics Research Unit \\
University of Luxembourg \\
6, rue Richard Coudenhove-Kalergi \\
L-1359 Luxembourg \\
Luxembourg}
\email{miguel.couceiro@uni.lu}
\author{Lucien Haddad}
\address[L. Haddad]{Math \& Info\\ Coll\`ege
militaire royal du Canada  \\ B.P 17000, STN Forces \\
Kingston ON, K7K 7B4 Canada.}
\email{haddad-l@rmc.ca}
\date{\today}
\begin{abstract}
Let $A$ be a finite non-singleton set. For $|A|=2$ we show that the partial clone consisting of all selfdual monotone partial functions  on $A$ is not finitely generated, while it is the intersection of two finitely generated maximal partial clones on $A$. Moreover for $|A| \ge 3$ we show that there are pairs of finitely generated maximal partial clones whose intersection is a non-finitely generated partial clone on $A$.

\end{abstract}
\maketitle

\section{Preliminaries}

Let $A$ be a finite set with cardinality $|A|:=k \ge 2$.
For a positive integer
$n$, an {\ $n$-ary partial function\/} on $A$ is a map $f: \d(f) \to
A$ where $\d(f)$ is subset of $A^n$ called the {\it domain} of $f$.
Let ${\rm Par}^{(n)}(A)$ denote the set of all $n$-ary
partial functions on $A$ and let 
$$\PA := \bigcup\limits_{n\ge 1} {\rm Par}^{(n)}(A).$$  
Moreover, denote by $ \OA$  the set of all {\it total} or everywhere defined functions, i.e., 
$$\OA := \bigcup\limits_{n\ge 1}\{f \in \PnA \st \d(f) = A^n\}.$$

 For $n, m \ge 1$, $f \in \PnA$ and $g_1, \dots, g_n
\in \PmA$,  the {\it composition} of $f$ and $g_1, \dots, g_n$,
denoted by $h:=f[g_1, \dots, g_n] \in \PmA$, is the partial function whose domain is 
 $$  \d(h):= \{\va \in A^m \st
 \va \in \bigcap\limits^n_{i=1} \d(g_i)\quad \mbox{and} \quad (g_1(\va), \dots, g_n(\va))\in \d(f)\}$$
and defined by
$$
h(\va):=f(g_1(\va),\dots, g_n(\va)),\quad \mbox{for all $\va \in \d(h)$.}
$$

 For every positive integer $n$  and each
$1 \le i \le n$, let
$e_{i}^n$ denote the {\it $n$-ary i-th projection} defined by
$\d(e^n_{i})=A^n$ and  $e^n_{i} (a_1, \dots, a_n) = a_i$ for all $(a_1, \dots, a_n) \in A^n$. Denote by $\JA$ the set of all projections on $A$.

 A {\it partial clone} on $A$ is a subset of $\PA$ closed under
composition and containing the set $\JA$ of all projections.
A  partial clone $C$ is said to be {\it strong} if it contains all
subfunctions of its functions, i.e., if for every $g \in \PA$,  we have $g \in C$ whenever $g=f|_{\d (g)}$, for some $f \in C$.

\begin{example}  Let $a \in A$ and consider the
set
$$ \mathcal{F}_a = \bigcup\limits_{n \ge 1} \{f \in \PnA \st 
\quad \mbox{if}\quad (a, \ldots, a) \in \d (f), \quad \mbox{then}\quad  f(a,\ldots, a) = a\}.
$$
Then $ \mathcal{F}_a$ is a strong partial clone on $A$.
\end{example}

The idea behind this example is formalized as follows.
For $h \ge 1$, let $\r$ be an $h$-ary relation on $A$ and $f$ be
an $n$-ary partial function on $A$. We say that $f~preserves~
\r$ if for every $h \times n$ matrix $M=[M_{ij}]$, whose columns
$M_{*j} \in \r,~(j=1,\ldots n)$ and whose rows $M_{i*} \in \d
(f)$ $(i=1,\ldots,h)$, we have
$(f(M_{1*}),\ldots,f(M_{h*})) \in \r$.  Define
$$\pp\r:=\{f
\in \PA \st f~{\rm preserves}~\r\}.$$

\begin{example} Denote by $\le$ and $\ne$ the two binary relations $\{(0,0),(0,1),(1,1)\}$ and  
$\{(0,1), (1,0)\}$, respectively, on $\2:=\{0,1\}$. Then
\begin{multline*}
\ple :=\{f \in \P2 \st (a_1,\dots,a_n)\in \d(f),~(b_1,\dots,b_n) \in \d(f) \quad \mbox{and}  \\
a_1 \le b_1, \dots, a_n \le b_n \imp f(a_1,\dots,a_n)\le f(b_1,\dots,b_n)\}
\end{multline*}
 and
\begin{multline*}
\pne:=\{f \in \P2 \st (a_1,\dots,a_n)\in \d(f) \quad \mbox{and}\\
(1+a_1,\dots,1+a_n) \in \d(f)\, \imp \, f(1+a_1,\dots,1+a_n)=1+f(a_1,\dots,a_n)\},
\end{multline*}
where $+$ denotes the sum modulo 2.
\end{example}

It is well known (see, e.g., \cite{bornerhaddad1,halau1} and \cite{Lau2006}, chapter 20) that $\pp
\r$ is a strong partial clone called the {\it partial clone determined by the relation} $\r$.  
Let $\PT (\r):= \pp(\r) \cap\OA$ be the partial clone consisting of all total functions preserving the relation $\r$.
The set of all partial clones on $A$, ordered by inclusion, forms an algebraic
lattice $\mathcal{L}_A$ where the meet coincides with the intersection.

We say that a partial clone $C_0$ on $A$ is
{\it covered}  by a partial clone
$C_1$ on $A$ if there is no partial clone $C$ such that $C_0 \subset C \subset C_1$.
A {\it maximal partial clone} is a partial clone covered by $\PA$.

 For $\cF \subseteq \PA$, let $\langle  \cF \rangle$ denote the {\em
partial clone generated by} $\cF$, i.e., the smallest partial clone containing $\cF$ or, equivalently, the intersection of
all partial clones on $A$ containing $\cF$. If $ C=\langle \cF \rangle$, then we say that $\cF$ is a {\it generating
set} for $ C$. A partial clone $C$ is said to be {\em finitely generated}
if it admits a finite generating set, i.e., if $ C=\langle \cF \rangle$ for
some finite set $\cF \subseteq  C$.

Generating sets for clones and partial clones have been extensively studied in the literature (see, e.g., surveys in \cite{ivosurvey} for the total case and \cite{bornerhaddad2} for the partial case). For instance,  Freivald \cite{freivald} showed that there are eight maximal partial clones on $\2:=\{0,1\}$,  and Lau \cite{lau2} showed that exactly two of  them are not finitely generated, namely, the two strong maximal partial clones of Slupecki type (see \cite{halau1} and \cite{Lau2006}, section 20, for details).

In this paper we are particularly interested in the two maximal partial clones $\pp(\le)$ and $\pp(\ne)$. As shown in \cite{haddad}, the interval of partial clones $[\ple \cap \pne, \P2]$ is infinite. However, it is still unknown if this interval is countably infinite or of continuum cardinality, and this problem seems to be difficult to decide. 

This lead the authors to study the partial clone  $\ple \cap \pne$. 
Here, we show that the partial clone $\ple \cap \pne$ is not finitely generated. 
We also generalize this result by showing that, on every non-singleton finite set $A$, there are pairs of finitely generated maximal partial clones whose intersection is a not finitely generated.

We shall make use the concept of ``separating clone" as introduced in \cite{bornerhaddad1} (see also \cite{bornerhaddad2,halau1,halau2}). 

\begin{definition} 
A clone $C$ on $A$ is {\it separating} if there exists $m \ge 1$ such that for all $n > m$ and all $\vb \in A^n$, there are $m$ functions $f_1,\dots,f_m \in C \cap  \OpnA$ such that for every $ \va, \vb\in A^n$, 
$$\mbox{if }\quad (f_1(\va),\dots,f_{m}(\va))=
(f_1(\vb),\dots,f_{m}(\vb)), \quad \mbox{then}  \quad \va=\vb.$$
\end{definition}

For example, $\OA$ is a separating clone on $A$ and it is shown in \cite{bornerhaddad1} that $\PT \varrho$ is a separating clone for every
$$\varrho \in \{\{0\},\{1\},\{(0,1)\},\{(0,1),(1,0)\},
\{(0,0),(0,1),(1,1)\} \}.$$
Note that  if $ C \subseteq  D$ are two clones and if $ C$ is a separating
clone, then $ D$ is also a separating clone. 

The main results of this paper are based on the following criterion established in \cite{bornerhaddad1} (see also \cite{bornerhaddad2,halau1,halau2}).

\begin{theorem}\label{criterion}
Let  $A$ be a finite set with at least two elements, and let $ D$ be a strong partial clone
on  $A$. 
If  $ C:= D\cap\OA$ is not a separating clone, then $ D$ is not finitely generated.
\end{theorem}

\section{Intersection of finitely generated maximal partial clones}

We start with the two element set $\2:=\{0,1\}$. It is shown in \cite{freivald}  that there are eight maximal partial clones on $\2$  and it is shown in \cite{lau2} that exactly two of  them are non-finitely generated  (they are the two strong maximal partial clones of Slupecki type, see \cite{halau1} and \cite{Lau2006} section 20 for details). In particular, Lau~\cite{lau2} showed the following interesting result.

 \begin{lemma} {\rm (\cite{lau2})}
  Let $A=\2$. Then, 
  \begin{enumerate}[(i)]
\item $\pp (\ne)$ is generated by its ternary partial functions, and
\item $\pp(\le)$ is generated by its binary partial functions.
\end{enumerate}
 \end{lemma}

Another proof of the fact that both  $\pp (\le)$ and $\pp (\ne)$ are finitely generated is given in Proposition 3.10 of \cite{bornerhaddad1}. The proof is based on the concept of separating clones and on the fact that every partial function in $\pp(\le)$ (in $\pp(\ne)$) can be extended to a total function in $\PT(\le)$ (in $\PT(\ne)$ respectively).  It is noteworthy that the clone  $\PT(\le) \cap \PT(\ne) $ on $\2$ is generated by the (total) ternary majority function
\[
\maj(x_1,x_2,x_3)= (x_1\wedge x_2) \vee (x_1\wedge x_3) \vee (x_2\wedge x_3)=(x_1\vee x_2) \wedge (x_1\vee x_3) \wedge (x_2\vee x_3).
\]

\begin{lemma} \label{trivial}
Let  $n \ge 3$ and let $f \in\PT(\le) \cap \PT(\ne)$ be an $n$-ary function on $\2$. Then,
  \begin{enumerate}[(i)]
\item  $f(0,\ldots,0)=0$, and
\item if $f(1,0,\ldots,0)=1$, then $f=e^n_1$ is the $n$-ary first projection function on $\2$.
\end{enumerate}
\end{lemma}
\begin{proof}  To see that (i) holds, suppose, for the sake of a contradiction, that $f(0,\ldots,0)=1$.
 Then, as $f \in \PT(\ne)$, we have $f(1,\ldots,1)=0$, and hence $f \not\in \PT(\le)$.

Now, we show that (ii) also holds. Since $f \in \PT(\le)$ and $f(1,0,\ldots,0)=1$, we have that $f(1,x_2,\ldots,x_n)=1$, for all $x_2,\ldots,x_n \in \2$. 
As $f\in \PT(\ne)$, we have that $f(0,x_2,\ldots,x_n)=0$, thus showing that $f=e^n_1$. 
\end{proof}

As mentioned above, it was shown in \cite{bornerhaddad1} that both $\PT(\le)$ and $\PT(\ne)$ are separating clones on $\2$. 
However, this is not the case for their intersection.

\begin{lemma} \label{not1}
The clone $\PT(\le) \cap \PT(\ne)$ is not a separating clone on $\2$.
\end{lemma}

\begin{proof} Let $m \ge 1$ and set $n:=m+1$. 
Let $f_1,\ldots,f_m \in \PT(\le) \cap \PT(\ne)$ be functions of arity $m+1$ and let $\vb ={\vec 0}:=(0,\ldots,0) \in \2^{m+1}$. 
Clearly, 
 $$(f_1({\vec 0}),\ldots,f_{m}({\vec 0}))=(0,\ldots,0).$$ 
 We show that there is $\va \in \2^{m+1}\setminus\{{\vec 0}\}$ such that
 $$(f_1({\va}),\ldots,f_{m}({\va}))=(0,\ldots,0).$$
 
 For $i=1,\ldots,m+1$, let $\ve_i$ be the vector corresponding to the $i$-th row of the identity $(m+1) \times (m+1)$ matrix, that is,
 $$\ve_i:=(0,\ldots,0,\underbrace{1}_{i{\rm-th~position}},0,\ldots,0)\in \2^{m+1}.$$
   Now, for every $j=1,\ldots,m$, there is at most one $i\in \{1,\ldots,m+1\}$ such that $f_j(\ve_i)=1$. 
 Indeed, by Lemma \ref{trivial} this is the case only if $f_j$ is the $(m+1)$-ary $i$-th projection function on $\2$. 
 Therefore, there is at least one 
 $t\in\{1,\ldots,m+1\}$ such that $f_j(\ve_t)=0$, for all $j\in\{1,\ldots,m\}$. 
 Set $\va = \ve_t$. Then 
 $$(f_1({\vec 0}),\ldots,f_{m}({\vec 0}))=(f_1({\ve_t}),\ldots,f_{m}({\ve_t})),$$ thus proving that
$\PT(\le) \cap \PT(\ne)$ is not a separating clone on $\2$. 
\end{proof}

Now, it is clear that
$$(\pp(\ne) \cap \pp(\le)) \cap \O2=\PT(\ne) \cap \PT(\le)$$
and thus, combining Theorem \ref{criterion} and Lemma \ref{not1}, we obtain the following result.

\begin{theorem}\label{th:two-not} The intersection of the two finitely generated maximal partial clones \break $\pp(\ne)$ and $\pp(\le)$ is not a finitely generated partial clone on $\2$. \endproof
\end{theorem}

We generalize Theorem~\ref{th:two-not} to any finite set $A=\k:=\{0,1,\ldots,k-1\}$ with $k \ge 3$. 
First, we need to recall a construction of \cite{pasz}.  
Let  $\le$ be the
natural (thus linear and bounded) order on $\k$, $p$ be a prime
divisor of $k$ and $\pi$ be the fixed-point-free permutation
defined by
$$ \pi=(\; 0\;\; 1\;\; \ldots \; p-1 \;)(\; p \;\; (p+1) \; \ldots \; (2p-1) \;) \ldots (\; (k-p) \; \; (k-p+1) \; \ldots \; (k-1) \;).$$ 

Note that the permutation $\pi$ consists  of  $k/p$ cycles of
length $p$. It is well known (see e.g., \cite{halau1, ivosurvey} and \cite{Lau2006} chapter 5) that both $\PT(\le)$ and $\PT (\pi)$ are
maximal clones on $\k$. Moreover, we have the following result appearing in \cite{pasz}. 

\begin{theorem} \label{palfymax} {\rm (\cite{pasz})}   Let $k \ge 3$. Then
$\PT(\le)   \cap  \PT(\pi) = \JA$. 
\end{theorem}

The partial clone $\pp \pi$ is not a maximal partial clone (see \cite{halau1,
halau2}) but is contained in a maximal partial clone. The following result is stated in
\cite{halau1} and its  proof is given in \cite{halau2}.

\begin{theorem} \label{lu-diet} {\rm (\cite{halau2})}~ Let $k \ge 3$, $p$ be a prime divisor of $k$ and $\pi$ be any fixed-point-free permutation with all cycles of length $p$ on $\k$. Define
$$\varrho_{\pi} := \{ (x,\pi(x),\pi^2(x),\ldots,\pi^{p-1}(x)) \st x\in \k\}.$$
Then,
\begin{enumerate}[(i)]
\item $\pp (\pi)  \subset \pp(\varrho_{\pi})$,
\item $\pp(\varrho_{\pi})$  is the unique maximal partial clone containing the maximal clone $\PT(\pi)$ on $\k$, and
\item $\pp(\varrho_{\pi})$ is generated by its unary and binary partial functions.
\end{enumerate}
\end{theorem}

It is known that for $k=8$, there is a  bounded partial order $\le _8$ on $\k$, for which the maximal partial clone $\PT(\le_8)$ is not finitely generated. This was shown by Tardos in \cite{tardos} for the partial order whose diagram is given below.

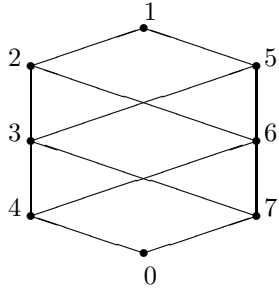
\parindent=1.5cm
\setlength{\unitlength}{0.50mm}
\begin{picture}(70,100)
\thinlines
\put(60,20){\circle*{2.0}}
\put(60,12){\makebox{$0$}}

\put(30,30){\circle*{2.0}}
\put(24,30){\makebox{$4$}}

\put(30,50){\circle*{2.0}}
\put(24,50){\makebox{$3$}}

\put(30,70){\circle*{2.0}}
\put(24,70){\makebox{$2$}}

\put(90,70){\circle*{2.0}}
\put(92,70){\makebox{$5$}}

\put(90,50){\circle*{2.0}}
\put(92,50){\makebox{$6$}}

\put(90,30){\circle*{2.0}}
\put(92,30){\makebox{$7$}}

\put(60,80){\circle*{2.0}}
\put(60,82){\makebox{$1$}}

\put(60,20){\line(-3,1){30}}
\put(60,20){\line(3,1){30}}
\put(30,30){\line(0,1){40}}
\put(90,30){\line(0,1){40}}
\put(30,70){\line(3,1){30}}
\put(90,70){\line(-3,1){30}}
\put(30,30){\line(3,1){60}}
\put(30,50){\line(3,1){60}}
\put(90,30){\line(-3,1){60}}
\put(90,50){\line(-3,1){60}}

\end{picture}
\parindent=0.5cm

Also, it is well known that every nontrivial order relation (whether it is bounded or not) determines a maximal partial clone on $\k$ (see e.g., \cite{halau1,halau2} and \cite{Lau2006} chapter 20) and such maximal partial clones are all finitely generated. 

\begin{theorem} \label{nozaki}{\rm (\cite{nozaki1, nozaki2})}
Let  $\le$ be a non-trivial order relation on $\k$. Then the maximal partial clone $\pp(\le)$ is generated by its binary partial functions on $\k$.
\end{theorem}

Combining Theorems \ref{criterion}, \ref{palfymax}, \ref{lu-diet} and \ref{nozaki} we obtain the following theorem which asserts that, for every $k\ge 1$, 
the partial clone $\pp(\varrho_{\pi})\cap \pp(\le)$ on $\k$ is not finitely generated. 

\begin{theorem} ~ Let $k \ge 3$,  $\varrho_{\pi}$ as in Theorem $\ref{lu-diet}$  and let  $\le$ be a linear order on $\k$. Then the intersection of the two finitely generated maximal partial clones  $\pp(\varrho_{\pi})$ and $\pp(\le)$ is not a finitely generated partial clone on $\k$.
\end{theorem}
\begin{proof} Clearly, $\pp(\le)   \cap  \pp (\varrho_{\pi})$ is a strong partial clone on $\k$, and we have
\[
\begin{array}{lll}
(\pp(\le)   \cap  \pp (\varrho_{\pi}) \cap \OA &=& (\pp(\le)   \cap  \OA) \cap (\pp (\varrho_{\pi}) \cap \OA) \\
&=& \PT(\le)  \cap \PT(\pi) \\
&=& \JA.
\end{array}\]
We show that $\JA$ is a not a separating clone on $\k$. Let $m \ge 1$,  set $n:=m+1$ and let $f_1,\ldots,f_m$ be projections on $\k$ of arity $m+1$. Set  $\vb ={\vec 0}:=(0,\ldots,0)$. Then 
$$(f_1({\vec 0}),\ldots,f_{m}({\vec 0}))=(0,\ldots,0).$$
Since there are $m+1$ different projections on $\k$ of arity $m+1$, there is $i \in \{1,2,\ldots,m,m+1\}$ such that the projection
 $e_i^{m+1}$ is not among $f_1,\ldots,f_m$. Choose 
 $$\va =\ve_i:=(0,\ldots,0,\underbrace{1}_{i{\rm-th~position}},0,\ldots,0) \in \k^{m+1}.$$ 
 Then $\va \ne {\vec 0}$ and  $(f_1({\va}),\ldots,f_{m}({\va}))=(0,\ldots,0)$, thus proving that $\JA$ is a not separating clone on $\k$. 
 By Theorem \ref{criterion}, $\pp(\le)   \cap  \pp (\varrho_{\pi})$ is not a finitely generated clone on $\k$.
\end{proof}

Note that  the partial clone $\pp(\le)   \cap  \pp (\varrho_{\pi})$ does not consist only of partial projections. 
Indeed, let $n \ge 3$ and construct the partial $n$-ary function $f$ on $\k$ by setting
$\d(f):=\{0\}\times\k^{n-1}$ and $f(\vv)=1$ for all $\vv=(0,x_2,\ldots,x_n)\in \{0\}\times\k^{n-1}$. 
Clearly, $f$ is not a partial projection. Moreover, $f \in \pp(\le)$ since it is a constant function, and $f \in \pp (\varrho_{\pi})$ since there is no $p \times n$ matrix over $\k$ with columns  in $\varrho_{\pi}$ and rows in $\d(f)$.

 The study in this paper yields the following interesting problem: 

\begin{problem} Let $|A|=k \ge 2$. Describe the strong partial clones on $A$ whose total part is $\JA$.
\end{problem}

 Note that Str$(\JA)$, the partial clone consisting of all subfunctions of projections, is one of them.
 These are not finitely generated partial clones.

\end{document}